\documentclass{article}
\usepackage{parskip,amsmath,amsthm,amsfonts,enumerate}
\usepackage[margin=1.3in]{geometry}

\begingroup
    \makeatletter
    \@for\theoremstyle:=definition,remark,plain\do{%
        \expandafter\g@addto@macro\csname th@\theoremstyle\endcsname{%
            \addtolength\thm@preskip\parskip
            }%
        }
\endgroup

\newtheorem{theorem}{Theorem}
\newtheorem{lemma}{Lemma}
\newtheorem{remark}{Remark}
\newtheorem{proposition}{Proposition}
\newtheorem{definition}{Definition}
\newtheorem{exam}{Example}

\newenvironment{keywords}{}{}

\renewcommand{\part}[2]{\subsubsection*{Part #1: {\normalfont #2.}} }

\newcommand{\ab}[1]{\operatorname{ab} #1}

\newcommand{\rann}[1]{\operatorname{Ann}_R #1}

\renewcommand{\hom}{\operatorname{Hom}}

\newcommand{\iso}{\simeq}

\newcommand\blfootnote[1]{%
  \begingroup
  \renewcommand\thefootnote{}\footnote{#1}%
  \addtocounter{footnote}{-1}%
  \endgroup
}

\title{Classification of algebras of level two in the variety of nilpotent algebras and Leibniz algebras}
\author{
    James Francese\thanks{
        Fullerton College,
        Fullerton, CA, US 92832,
        jamesafrancese@gmail.com
    }
    \and
    Abror Khudoyberdiyev\thanks{
        National University of Uzbekistan, Institute of Mathematics,
        Tashkent, Uzbekistan 100125,
        khabror@mail.ru
    }
    \and
    Bennett Rennier\thanks{
        University of Oklahoma,
        Norman, OK, US 73019,
        bennett@brennier.com
    }
    \and
    Anastasia Voloshinov\thanks{
        Pomona College,
        Claremont, CA, USA 91711,
        nastiav13@gmail.com
    }
    \renewcommand\footnotemark{}
    \blfootnote{This work was supported by NSF award number 1658672.}
}

\begin{document}
\newpage
\maketitle

\begin{abstract}
    \noindent{}This paper is devoted to the description of complex finite-dimensional algebras of level two.
    We obtain the classification of algebras of level two in the variety of Leibniz algebras. It is shown
    that,  up to isomorphism, there exist three Leibniz algebras of level two, one of which is solvable, and
    two of which are nilpotent. Moreover we describe all algebras of level two in the variety of nilpotent
    algebras.
\end{abstract}


\section{Introduction and Background}

The theory of deformations and degenerations of algebras has its origins in certain formal relations between
physical theories, and has become a lively subject in algebraic and differential geometry, as well as
noncommutative and nonassociative algebra. It was a very influential precept of Niels Bohr that a new
physical theory, which is supposed to ontologically overlap with a previously accepted theory, should
somehow yield the old theory as a limiting or special case \cite{corrPrin}. This is a statement of his
``correspondence principle", which is realized in quantum mechanics via the limit of the Moyal bracket:
\begin{displaymath}
    [f, g ] = \{f, g\} + \mathcal{O}(\hbar ^2)
\end{displaymath}
as $\hbar \rightarrow 0$. Here the bracket $\{ ~ , ~ \} = \partial^i \partial_i - \partial_i \partial ^i $
is the Poisson bracket of classical mechanics, being a sum of commutators for first-order differential
operators. In this way classical mechanics emerges, as a limiting case, from quantum mechanics for small
values of $\hbar$. This realization motivated the modern deformation theory of algebras, which originated
with Gerstenhaber \cite{ringDeformations} and others. This theory has powerful applications in the
classification of algebraic varieties and the quantization theory of Poisson manifolds, where the physical
meaning of deformation remains especially explicit \cite{quant, germdef}. There we have a manifold with
tangent space $V$ equipped with a Poisson bivector $\mathfrak{p} \in V \wedge V^{*}$, and an algebra of
observables $f, g \in C^{\infty} (V)$. We then define a product of these observables as a power series in a
parameter $\hbar$:
\begin{displaymath}
    (f \star g) (\mathfrak{p}):= f \cdot g + \hbar \mathfrak{p}^i _j \frac{\partial f}{\partial x^i}
    \frac{\partial g}{\partial x_j} + \frac{1}{2} \hbar ^2 \mathfrak{p}^k _{\ell} \mathfrak{p}^i _j
    \frac{\partial ^ 2 f}{\partial x^k \partial x^i } \frac{\partial ^ 2 g }{\partial x_{\ell} \partial x
    _j} + \ldots
\end{displaymath}
where we have made use of the summation convention. Defining the bracket
\begin{displaymath}
    [f, g ] := -i\hbar ^{-1} (f \star g - g \star f),
\end{displaymath}
in the ordinary quantum mechanical case for a single state (i.e. where all differential operators are
first-order) we recover the above limit, relating the Moyal to the Poisson bracket. The full higher-order
approach is currently being used in the study of formal quantum field theory \cite{AQFT, aquant},
a version of Hilbert's 6th problem \cite{hil}.

For this paper, we consider the case of a finite-dimensional algebra over a closed field \( k \), with $k =
\mathbf{C}$ being of special focus. In general, an algebra over $k$ can be considered as an element \( \mu
\in \hom(V \otimes V, V) \), where \( V \) is an \(n\)-dimensional \(k\)-vector space. Thus, in the purely
abstract deformation theory we consider algebras $(A, \mu), (A_t, \mu_t) \in V \otimes k [ [t]]$ related by
a formal power series:
\begin{displaymath}
    \mu_t = \mu + \sum_{i=1} ^{\infty} t^i \varphi_i \hspace{0.3cm} \text{where} \hspace{0.2cm} \varphi_{i} \in \hom(V \otimes V, V)
\end{displaymath}
so that in the linear case $\mu_{t, 1} = \mu + t \varphi_1$ we study algebras with multiplication differing by
a $2$-coboundary.\footnote{An abstract ``Poisson bracket" on an algebra is always available in the case of a
linear deformation, since we can define $\{ x, y \} := \frac{1}{2}(\mu_{t, 1}(x, y) - \mu_{t, 1} (y, x) )$.}
Kodaira and Spencer gave the original idea of infinitesimal deformations for complex analytic manifolds
\cite{infinitesimalDeformations}. Most notably, they proved that infinitesimal deformations can be
parametrized by a related cohomology group. In fact, cohomology detects deformation at all orders, and it is
therefore unsurprising that one can develop the deformation theory in any abelian category \cite{abcat}.

Prior to the development of deformation theory, it had already been realized that the space and time
symmetries of Newtonian mechanics were recovered in the $c \rightarrow \infty$ limit of special relativity,
were $c$ is the speed of light. In that case the Lie algebra of the Poincare group \textit{degenerates} to
the Lie algebra of the Galilean group, an observation first made by Wigner and \.{I}n\"{o}n\"{u} in
\cite{Wignercon}. This is a geometric process, and can be defined purely in terms of the Zariski topology on
\( \hom(V \otimes V,V) \).

In the finite-dimensional case over a field of characteristic zero, such degenerations can be described directly in terms of the singular limit of a linear group action. Let $\text{Alg}_n (k)$ be the variety of $n$-dimensional algebras over $k$, and let $(A, \lambda), (A, \mu) \in \text{Alg}_n (k)$. Define an action on $\text{Alg}_n (k)$ by means of
\begin{displaymath}
    (g \ast \mu) (x, y) := g( \lambda (g^{-1}(x) , g^{-1}(y) )) \hspace{0.5cm} \text{where} \hspace{0.2cm} g
    \in \text{GL}_n (k), \hspace{0.2cm} x,y \in A
\end{displaymath}
which just represents a change of basis for $A$ as an algebra. Thus the orbit of the algebra $(A, \lambda)$ under this action is given by
\begin{displaymath}
    \text{Orb} (A) := \{ L \in \text{Alg}_n (k) \hspace{0.1cm} | \hspace{0.1cm} L \simeq A \}.
\end{displaymath}

\begin{definition}  An algebra $(A, \lambda)$ is said to \textit{degenerate} to the algebra $(A,\mu)$ if $\text{Orb}(A,\mu) \subseteq \overline{ \text{Orb}(A, \lambda)} $. We write $  \lambda \rightarrow \mu $ to denote this degeneration.
\end{definition}

In the case $k = \mathbb{C}$, we have that $\lambda \rightarrow \mu$ if and only if there is a $g_t \in GL_n (\mathbb{C}(t))$ such that $\forall x, y \in A$,
\begin{displaymath}
    \mu(x, y) = \lim\limits_{t \rightarrow 0} g_t (\lambda (g_t^{-1}(x),g_t^{-1}(y))).
\end{displaymath}

We call a degeneration $\lambda \rightarrow \mu$  \textit{trivial} if $(A, \lambda) \simeq (A, \mu)$, and \textit{direct} if it is non-trivial, and there is no algebra $(A, \nu)$ such that $\lambda \rightarrow \nu \rightarrow \mu$. If $\lambda \rightarrow \mu$, then $\lambda$ is a non-trivial deformation of $\mu$, thus it is common to pass from the degeneration theory to the deformation theory.

It is clear that every non-abelian algebra in $\text{Alg}_n (\mathbf{C})$ degenerates non-trivially to the abelian algebra $\ab_{n}$, but of course not all such degenerations will be direct; the distance of an algebra from $\ab_{n}$, in terms of the degeneration theory, is given by its \textit{level}.

\begin{definition}
    The level of an algebra \(\lambda\) is the maximum length of a chain of direct degenerations to
    $\ab_{n}$. We denote the level of an algebra by \( \text{lev}_{n}(\lambda) \).
\end{definition}

Concerning algebras of level one, we have the following result proved by Khudoyberdiyev and Omirov \cite{levelone}.

\begin{theorem}\label{omi-khudo}
    Let \( A \) be an algebra of level one. Then \( A \) is isomorphic to one of the following pairwise
    non-isomorphic algebras:
    \begin{alignat*}{3}
        p_{n}^{-} &: \qquad && e_1 e_{i} = e_{i}, \qquad &&e_{i} e_{1} = - e_{i}, \qquad 2 \leq i \leq n; \\
        n_{3}^{-} \oplus \ab_{n - 3} &: \qquad &&e_1 e_2 = e_3, \qquad &&e_2 e_1 = - e_3; \\
        \lambda_{2} \oplus \ab_{n - 2} &: \qquad &&e_1 e_1 = e_2; \\
        \nu_{n}(\alpha) &: \qquad &&e_1 e_1 = e_1, \qquad &&e_1 e_{i} = \alpha e_{i}, \qquad e_{i} e_1 =
        (1 - \alpha) e_{i}, \qquad2 \leq i \leq n.
    \end{alignat*}
\end{theorem}

The level two case, within the varieties of Lie, Jordan, and associative algebras, has been resolved by Khudoyberdiyev in \cite{leveltwo}. In particular, that paper provides the following theorem.

\begin{theorem}\label{khudo}
    Let \( G \) be a Lie algebra of level two. Then \( G \) is isomorphic to one of the following pairwise
    non-isomorphic algebras:
    \begin{alignat*}{3}
        n_{5, 1} \oplus \ab_{n-5}&: \qquad && e_1 e_3 = e_5, \qquad &&e_2 e_4 = e_5, \qquad 2 \leq i \leq n; \\
        n_{5, 2} \oplus \ab_{n - 5} &: \qquad &&e_1 e_2 = e_4, \qquad &&e_1 e_3 = e_5; \\
        r_{2} \oplus \ab_{n - 2} &: \qquad &&e_1 e_1 = e_2; \\
        g_{n, 1}(\alpha) &: \qquad &&e_1 e_2 = \alpha e_2, \qquad &&e_1 e_{i} = e_{i}, \qquad 3 \leq i \leq n, \alpha \in \mathbf{C}/\{0, 1 \};\\
        g_{n, 2} &: \qquad  && e_1 e_2 = e_2 + e_3, \qquad && e_1 e_i, \qquad \hspace{0.75cm} 3 \leq i \leq n.
            \end{alignat*}
\end{theorem}

It is still desirable to obtain a complete classification of level two algebras. One step is to ask about the existence of level two algebras in other varieties. Since $n_{5, 1}$ and $n_{5, 2}$ are Lie, we may well ask if there are any non-Lie \textit{Leibniz} algebras of level two.

\begin{definition}\cite{loday}
    A (right) Leibniz algebra is an non-associative algebra such that for all \( x,y,z \in L \), the
    following identity holds:
    \begin{displaymath}
        x (yz) = (xy) z - (xz) y.
    \end{displaymath}
\end{definition}

This is a natural generalization of Lie algebras, in that an antisymmetric Leibniz algebra is Lie.
Note that a left Leibniz algebra is defined by  identity $(x y) z = x(y z) - y (x z).$

Degenerations of Lie and Leibniz algebras were the subject of
numerous papers, see for instance \cite{nilpotentLeibnizFive, beltita, Burde, condeg, nilLie} and references given therein, and their research continues
actively. In particular, in \cite{solvLeib, Leibsuper} some irreducible components of Leibniz algebras are found.

 In this paper, we extend Theorem \ref{khudo} to identify all non-Lie Leibniz algebras of level two; we find that two of these are nilpotent and one is solvable. We then proceed to classify all $n$-dimensional nilpotent algebras of level two, and find that these are all Leibniz.

\section{Main Results}

Our first main result is the classification of Leibniz algebras of level two.

\begin{theorem}\label{mainProof}
    Let \(L\) be a \(n\)-dimensional non-Lie Leibniz algebra of level two. Then \( L \) is isomorphic one of the following three algebras:
    \begin{alignat*}{3}
        L_{4}(\alpha)\oplus a_{n-3} &: \qquad &&e_1 e_1 = e_3, \qquad &&e_2 e_1 = e_3, \qquad e_2 e_2 =
        \alpha e_{3}; \\
        L_{5}\oplus a_{n-3} &: \qquad &&e_1 e_1 = e_3, \qquad &&e_1 e_2 = e_3, \qquad e_2 e_1 = e_3; \\
        r_{n} &: \qquad &&e_i e_{1} = e_{i}, \qquad &&2 \leq i \leq n.
    \end{alignat*}
\end{theorem}

Together with the nilpotent Lie algebras of level two identified in \cite{leveltwo}, our other main result identifies these four algebras as the only nilpotent algebras of level two.

\begin{theorem}\label{nilpotentclass}
Any finite-dimensional nilpotent algebra of level two is isomorphic to one of the following algebras:
$$\begin{array}{rlll}
    n_{5, 1} \oplus a_{n-5}  : & e_1 e_3 = e_5, &  e_2e_1 = e_5;\\
    n_{5, 2} \oplus a_{n-5} : & e_1 e_2 = e_4 , & e_1 e_3 = e_5;\\
    L_4 (\alpha) \oplus a_{n-3}  :& e_1 e_1 = e_3, & e_2 e_2 = \alpha e_3, & e_1 e_2 = e_3;\\
    L_5 \oplus a_{n-3} : & e_1 e_1 = e_3 , & e_1 e_2 = e_3, & e_2 e_1 = e_3.
\end{array}
$$
\end{theorem}

It is now natural to ask if any algebra of level two is a direct sum of two level one algebras. The
following examples gives us a negative answer to this question.

\begin{exam}\label{example1} The algebras
\[\begin{array}{rlll} n^{-}_{3}\oplus \lambda_{2}=\{x_1, x_2, x_3, x_4, x_5
\}:& x_1 x_1 = x_2,& x_3x_4= x_5, & x_4x_3 = - x_5;\\
\lambda_{2}\oplus \lambda_{2}=\{x_1, x_2, x_3, x_4\}: & x_1x_1 =
x_2, &  x_3x_3= x_4; \\ \lambda_{2}\oplus p_{n}^{-}=\{x_1, x_2, x_3,
x_4, \dots, x_n\}: & x_1x_1 = x_2, & x_ix_3= x_i, &  x_3x_i = -x_i,
\ \ 4 \leq i \leq n.
\end{array}\]
via the family of matrices
\[\left\{\begin{array}{ll}g_t^{-1}(x_1)=t(x_1+x_3), \\ g_t^{-1}
(x_2)=\frac t 2 (x_1+x_4),\\ g_t^{-1}(x_3)=t^{2}x_2,\\ g_t^{-1}(x_4) =
t^{2}x_1,\\ g_t^{-1}(x_5)=t(x_5 + x_2), \end{array}\right. \
\left\{\begin{array}{ll} g_t^{-1}(x_1)=tx_1, \\
g_t^{-1} (x_2)= t (x_1+x_3),\\ g_t^{-1}(x_3)=t^{2}x_2,\\ g_t^{-1}(x_4)
= t(x_2+x_4),\end{array}\right. \  \left\{\begin{array}{ll}g_t^{-1}(x_1)= t(x_1+x_4),\\ g_t^{-1}(x_2)=t(\frac 1
2 x_1+x_3),\\ g_t^{-1}(x_3)=t^2x_2,\\ g_t^{-1}(x_4) = t(x_4+\frac 1
2x_2), \\ g_t^{-1}(x_i) = tx_i, & 5 \leq i \leq n,\end{array}\right.\]
degenerate to the algebras $L_4(\frac{1}{4})\oplus a_2,$ $L_5\oplus
a_1$ and $L_4(\frac{1}{4})\oplus a_{n-3}$ respectively.
\end{exam}

Since the algebras \( L_{4}(\alpha) \) and $L_5$ are not algebras of level one, we deduce that the level of
the algebras $n^{-}_{3}\oplus \lambda_{2},$ $\lambda_{2}\oplus \lambda_{2}$ and $\lambda_{2}\oplus
p_{n}^{-}$ must be greater than two.

Now let $L$ be a $n$-dimensional complex algebra and $\{e_1, e_2, \dots, e_n\}$ be a basis of $L.$ The
multiplication on the algebra $L$ is defined by the products of the basis elements; namely, by the products
\[e_i e_j=\sum\limits_{k=1}^n\gamma_{i,j}^ke_k,\] where \(\gamma_{i,j}^k \) are the structural constants.

We first prove a very useful lemma, which will allow us to immediately conclude a degeneration to either \(
L_{4}(\alpha) \) or \( L_{5} \) based on a multiplication table of a certain form.

\begin{lemma}\label{nilpotentLeibniz}
    Suppose \( L \) is an \(n\)-dimensonal algebra and let $\{e_1, e_2, \dots, e_n\}$ be a basis of $L.$ If
    there exist distinct \( i, j,k \) such that
    \begin{displaymath}
        (\gamma_{i,i}^k, \gamma_{i,j}^k, \gamma_{j,i}^k, \gamma_{j,j}^k)
        \notin \{(0, \beta, -\beta, 0), (\delta, \beta, \beta, \frac {\beta^2} {\delta})\} \hspace{0.2cm}
        \text{where} \hspace{0.2cm} \delta \neq 0,
    \end{displaymath}
    then \( L \to L_{4}(\alpha)\) or \( L \to L_{5} \).
\end{lemma}
\begin{proof}
    Without loss of generality, we may assume \( i = 1 \), \( j = 2 \), and \( k = 3 \). We see that if we take the degeneration
    \begin{displaymath}
        g_{t}(e_{1}) = t^{-1} e_1, \quad g_{t}(e_2) = t^{-1} e_2, \quad g_{t}(e_{i}) = t^{-2} e_{i} \quad 3 \leq i\leq n,
    \end{displaymath}
    then we have following nontrivial products
    \begin{align*}
        &e_{1} e_1 = \gamma_{1,1}^3 e_3 + \sum_{s=4}^n\gamma_{1,1}^s e_s, \qquad  e_{1} e_2 =
        \gamma_{1,2}^3 e_3 + \sum\limits_{s=4}^n\gamma_{1,2}^s e_s,\\
        &e_{2} e_1 = \gamma_{2,1}^3 e_3 + \sum\limits_{s=4}^n\gamma_{2,1}^s e_s, \qquad e_{2} e_2 =
        \gamma_{2,2}^3 e_3 + \sum\limits_{s=4}^n\gamma_{2,2}^s e_s.
    \end{align*}
        Furthermore, if we take the additional degeneration
    \begin{displaymath}
        g_{t}(e_{1}) = t^{-1} e_1, \quad g_{t}(e_2) = t^{-1} e_2, \quad g_{t}(e_3) = t^{-2} e_3, \quad g_{t}(e_{i}) = t^{-1} e_{i} \quad 4 \leq i\leq n,
    \end{displaymath}
    then we have an algebra with the following multiplication
    \begin{displaymath}
    e_{1} e_{1} = \gamma_{1,1}^3 e_3, \quad e_{1} e_{2} = \gamma_{1,2}^3 e_3, \quad e_{2} e_{1} =
        \gamma_{2,1}^3 e_3, \quad e_{2} e_{2} = \gamma_{2,2}^3 e_3.
    \end{displaymath}

    We see that this algebra is nilpotent and also non-Lie as \((\gamma_{1,1}^3, \gamma_{1,2}^3,
    \gamma_{2,1}^3, \gamma_{2,2}^3) \neq  (0, \beta, -\beta, 0)\). Moreover, since \((\gamma_{1,1}^3,
    \gamma_{1,2}^3, \gamma_{2,1}^3, \gamma_{2,2}^3) \neq  (\delta, \beta, \beta, \frac {\beta^2}
    {\delta})\), we conclude that $L$ is not isomorphic to the algebra $\lambda_2.$ Due to the
    classification of three dimensional nilpotent Leibniz algebras \cite{3-dim}, we conclude that this
    algebra is isomorphic to  either \( L_{4}(\alpha) \) or \( L_{5} \).
\end{proof}

\subsection{Classification of algebras of level two in the variety of Leibniz algebras}

Let $L$ be a Leibniz algebra and let \( x \in L \). We define \( \varphi_{x} : L \to L \) to be the linear
operator where \( \varphi_{x}(y) = yx + xy \). We see that by applying the Leibniz identity we get the
following two equations:
\begin{align*}
    &z \varphi_{x}(y) = z (yx + xy) = z (yx) + z (xy) = (zy)x - (zx)y + (zx)y - (zy)x = 0,\\
    &z (xx) = (zx) x - (zx) x = 0.
\end{align*}
This proves that both \( \varphi_{x}(y) \) and \( xx \) are in the right annihilator for any \( x \in L \).

In this section we will examine the matrix representation of \( \varphi_{x} \) on a case-by-case basis in
order to prove Theorem \ref{mainProof}.

\begin{proposition}\label{prop1}
    Let \(L\) be a \(n\)-dimensional non-Lie Leibniz algebra which is not of level one (i.e. \( L
    \not\iso \lambda_{2} \)). Then \( L \) degenerates to one of the following three algebras:
    \begin{alignat*}{3}
        L_{4}(\alpha)\oplus a_{n-3} &: \qquad &&e_1 e_1 = e_3, \qquad &&e_2 e_1 = e_3, \qquad e_2 e_2 =
        \alpha e_{3}; \\
        L_{5}\oplus a_{n-3} &: \qquad &&e_1 e_1 = e_3, \qquad &&e_1 e_2 = e_3, \qquad e_2 e_1 = e_3; \\
        r_{n} &: \qquad &&e_i e_{1} = e_{i}, \qquad &&2 \leq i \leq n.
    \end{alignat*}
\end{proposition}

\begin{proof}
    Since \(L\) is a non-Lie Leibniz algebra, we know that there exists an element \( x \in L \) such that
    \( xx \neq 0 \). Suppose that \( xx = \alpha x \) for some constant \( \alpha \). Then, by the Leibniz
    identity, we have that
    \begin{displaymath}
        \alpha^2 x = \alpha xx = x (\alpha x) = x (xx) = (xx) x - (xx) x = 0
    \end{displaymath}
    which means that \( \alpha = 0 \). This is a contradiction, though, as \( xx \neq 0 \). Thus, it must be
    that \( xx \) is linearly independent from \(x\). Using this, we can form a basis \( \{ e_1 , e_2 ,
    \dots , e_{n} \} \), where \( e_1 = x \) and \( e_2 = xx \).

    We define the linear operator \( \varphi = \varphi_{x} \) and let \( (\alpha_{i,j}) \) be its matrix
    form. Thus, we have that
    \begin{displaymath}
        \varphi =
        \begin{bmatrix}
            0 & \alpha_{1, 2} & \alpha_{1, 3} & \alpha_{1, 4} & \dots & \alpha_{1, n} \\
            2 & \alpha_{2, 2} & \alpha_{2, 3} & \alpha_{2, 4} & \dots & \alpha_{2, n} \\
            0 & \alpha_{3, 2} & \alpha_{3, 3} & \alpha_{3, 4} & \dots & \alpha_{3, n} \\
            0 & \alpha_{4, 2} & \alpha_{4, 3} & \alpha_{4, 4} & \dots & \alpha_{4, n} \\
            \vdots & \vdots & \vdots & \vdots & \ddots & \vdots \\
            0 & \alpha_{n, 2} & \alpha_{n, 3} & \alpha_{n, 4} & \dots & \alpha_{n, n}
        \end{bmatrix}
    \end{displaymath}

    Suppose that \( \alpha_{j,k} \neq 0 \) for some \( 1,j,k \) distinct and \( k \geq 3 \). This means that
    we have the following products
    \begin{alignat*}{2}
        &e_1 e_1 = e_2, \qquad && e_{j} e_{j} = \gamma_{j,j}^{k} e_{k} + \sum_{s = 1, s \neq k}^{n}
        \gamma_{j,j} ^{s} e_{s},\\
        &e_1 e_j = \gamma_{1,j}^k e_{k} + \sum_{s=1, s \neq k}^n \gamma_{1,j}^s e_{s}, \qquad &&e_{j} e_1 =
        (\alpha_{j,k} - \gamma_{1,j}^k) e_{k} + \sum_{s=1, s \neq k}^n \gamma_{j,1}^s e_{s}.
    \end{alignat*}

    Since \( \gamma_{1,1}^{k} = 0 \) and \( \gamma_{1,j}^{k} \neq - (\alpha_{j,k} - \gamma_{1,j}^{k}) \), we
    can apply Lemma \ref{nilpotentLeibniz} on the indices $1, j, k$ to see that \( L \to L_{4}(\alpha) \) or \(
    L_{5}\). Now we focus the case where \( \alpha_{j,k} = 0 \) for \( 1,j,k \) distinct and \( k \geq 3 \)
    . This gives us the following matrix representation of \(\varphi\):
    \begin{displaymath}
        \varphi =
        \begin{bmatrix}
            0 & \alpha_{1, 2} & \alpha_{1, 3} & \alpha_{1, 4} & \dots & \alpha_{1, n} \\
            2 & \alpha_{2, 2} & \alpha_{2, 3} & \alpha_{2, 4} & \dots & \alpha_{2, n} \\
            0 & 0 & \alpha_{3, 3} & 0 & \dots & 0 \\
            0 & 0 & 0 & \alpha_{4, 4} & \dots & 0 \\
            \vdots & \vdots & \vdots & \vdots & \ddots & \vdots \\
            0 & 0 & 0 & 0 & \dots & \alpha_{n, n}
        \end{bmatrix}
    \end{displaymath}

    Suppose that \( \alpha_{2,i} \neq 0 \) for some \( i \neq 1,2 \). Without loss of generality, let \(
    i = 3 \). Then if we take the change of basis \( e_{3} ' = e_1 - \frac{2}{\alpha_{2,3}} e_3 \), we have
    that
    \begin{align*}
        \varphi(e_3 ') &= \varphi(e_1 - \frac{2}{\alpha_{2,3}} e_3) = \varphi(e_1) - \frac{2}{\alpha_{2,3}
        } \varphi(e_3) = 2 e_2 - \frac{2}{\alpha_{2,3}} (\alpha_{1,3} e_1 + \alpha_{2,3} e_2 +
        \alpha_{3,3} e_3)\\
        &= - \frac{2 \alpha_{1,3}}{\alpha_{2,3}} e_1 - \frac{2 \alpha_{3,3}}{\alpha_{2,3}} e_3 =
        - \frac{2 \alpha_{1,3}}{\alpha_{2,3}} e_1 - \alpha_{3,3} (e_1 - e_3 ') =
        - \left( \frac{2 \alpha_{1,3}}{\alpha_{2,3}} - \alpha_{3,3} \right) e_1 + \alpha_{3,3} e_3'.
    \end{align*}
    Thus, we may assume that \( \alpha_{2,3} = 0 \). Since \( i= 3 \) was arbitrary, we can therefore assume
    that \( \alpha_{2,i} = 0 \) for \( 3 \leq i \leq n \). This means that the matrix representation of
    \(\varphi\) has the form:
    \begin{displaymath}
        \varphi =
        \begin{bmatrix}
            0 & \alpha_{1 2} & \alpha_{1 3} & \alpha_{1 4} & \dots & \alpha_{1 n} \\
            2 & \alpha_{2 2} & 0 & 0 & \dots & 0 \\
            0 & 0 & \alpha_{3 3} & 0 & \dots & 0 \\
            0 & 0 & 0 & \alpha_{4 4} & \dots & 0 \\
            \vdots & \vdots & \vdots & \vdots & \ddots & \vdots \\
            0 & 0 & 0 & 0 & \dots & \alpha_{n n}
        \end{bmatrix}
    \end{displaymath}

    Now suppose there exists \( j \geq 3 \) such that \( \alpha_{2,2} \neq \alpha_{j,j} \).
     Then we see that if we take the basis change \( e_{j}' = e_{2} + e_{j}
    \), we have that
    \begin{align*}
        \varphi(e_{j}') = \varphi(e_{2}) + \varphi(e_{j}) &= (\alpha_{1,2} + \alpha_{1,j} ) e_1 + \alpha_{2,2} e_{2} + \alpha_{j,j} e_{j} \\
        &= (\alpha_{1,2} + \alpha_{1,j} ) e_1 +
         (\alpha_{2,2} - \alpha_{j,j}) e_2 + \alpha_{j,j} e_{j}'
    \end{align*}
       By reapplying our previous argument, we see that it must be that either \( L \to L_{4}(\alpha), L_{5} \)
    or \( \alpha_{2,2} - \alpha_{j,j} = 0 \). Thus, we have that \( \alpha = \alpha_{2,2} = \alpha_{i,i} \)
    for \( 3 \leq i \leq n \).

    Now suppose that \( \alpha \neq 0 \). We can then take the basis change \( e_{i}' = \varphi(e_{i}) \)
    for \( 2 \leq i \leq n\). This means that \( e_2 , e_3, \dots , e_{n} \in \rann(L) \) and thus that \(
    \varphi(e_{i}) = e_{i} e_{1} \) for \( 2 \leq i \leq n \). Moreover, if \( \alpha_{1,i} \neq 0 \) for
    some \( 2 \leq i \leq n \), then we'd have that
    \begin{displaymath}
        \alpha_{1i} e_1 = \varphi(e_{i}) - \alpha e_i \in \rann(L)
    \end{displaymath}
    which is a contradiction to the fact that \( e_1 e_1 = e_2 \). Therefore, we have an algebra \( L
    \) with the following multiplication table:
    \begin{displaymath}
        e_1 e_1 = e_2, \qquad e_{i} e_1 = \alpha e_{i}, \ (\alpha \neq 0), \qquad 2 \leq i \leq n.
    \end{displaymath}
    If we then take the basis transformation \( e_1 ' = \frac{1}{\alpha} (e_1 - \frac{1}{\alpha} e_2) \), we
    see that this is exactly the algebra \( r_{n} \).

    Suppose then that \( \alpha = 0 \). Again, we see that if \( \alpha_{1,i} \neq 0 \) for some \( 2 \leq i
    \leq n \), then we'd have that \( \alpha_{1,i} e_1 = \varphi(e_{i}) \in \rann(L) \), which is a
    contradiction. This means that \(\varphi\) has the following matrix representation:
    \begin{displaymath}
        \varphi =
        \begin{bmatrix}
            0 & 0 & 0 & 0 & \dots & 0 \\
            2 & 0 & 0 & 0 & \dots & 0 \\
            0 & 0 & 0 & 0 & \dots & 0 \\
            0 & 0 & 0 & 0 & \dots & 0 \\
            \vdots & \vdots & \vdots & \vdots & \ddots & \vdots \\
            0 & 0 & 0 & 0 & \dots & 0
        \end{bmatrix}
    \end{displaymath}

    We now consider following products
    \begin{displaymath}
        e_{2} e_{2} = 0, \quad e_{3} e_{2} = 0,\quad e_{2} e_{3} = \gamma_{2,3}^2 e_{2}+ \gamma_{2,3}^3
        e_{3}+ \sum_{s=4}^n \gamma_{2,3}^s e_{s}, \quad e_{3} e_{3} = \gamma_{3,3}^2 e_{2}+ \gamma_{3,3}^3
        e_{3}+ \sum_{s=4}^n \gamma_{3,3}^s e_{s}.
    \end{displaymath}

    If there exist $\gamma_{2,3}^s \neq 0,$ for some $4 \leq s \leq n,$ then by Lemma
    \ref{nilpotentLeibniz}, we have that \( L \to L_{4}(\alpha), L_{5} \). Thus, we assume that
    $\gamma_{2,3}^s = 0$ for $4 \leq s \leq n.$ Therefore, \( e_{2} e_3 = \gamma_{2,3}^2 e_{2}+
    \gamma_{2,3}^3 e_{3} \).  If \( \gamma_{2,3}^3 \neq 0 \), then we have that \( \gamma_{2,3}^3 e_3 = e_2
    e_3 - \gamma_{2,3}^2 e_{2} \in \rann(L) \). This means that \( e_2 e_3 = 0 \), a contradiction as
    \(\gamma_{2,3}^3 \neq 0 \). Hence, \( e_2 e_3 = \gamma_{2,3}^2 e_{2} \).

    Assume that \( \gamma_{2,3}^{2} \neq 0 \). If we take the change of basis \( e_2 ' = e_2 - Ae_1 \), then
    we have that
    \begin{align*}
        &e_2 ' e_2 ' = (e_2 - Ae_1)(e_2 - Ae_1) = A^2 e_1 e_1 = A^2e_2 = A^3 e_1 + A^2 e_2',\\
        &e_2' e_3 = (e_2 - A e_1) e_3 = \gamma_{2,3}^{2} e_2 - A e_1 e_3= A (\gamma_{2,3}^2- \gamma_{1,3}^1
        - A\gamma_{1,3}^2) e_1 + (\gamma_{2,3}^2- A\gamma_{1,3}^2)e_2' -A\sum_{s = 3}^{n} \gamma_{1,3}^{s}
        e_{s},\\
        &e_3 e_2 ' = e_3 (e_2 - A e_1) =  Ae_1 e_3 =   A (\gamma_{1,3}^1 + A\gamma_{1,3}^2) e_1 +
        A\gamma_{1,3}^2e_2' + A\sum_{s = 3}^{n} \gamma_{1,3}^{s} e_{s},\\
        &e_3 e_3 = \gamma_{3,3}^{1} e_1 + \gamma_{3,3}^{2} e_1 + \sum_{s = 3}^{n} \gamma_{3,3}^{s} e_{s}.
    \end{align*}

    Since \( \gamma_{2,3}^{2} \neq 0 \), we see that we can always choose \( A \) so that $\gamma_{2,3}^2-
    \gamma_{1,3}^1 - A\gamma_{1,3}^2 \neq \gamma_{1,3}^1 + A\gamma_{1,3}^2.$ Hence, using Lemma
    \ref{nilpotentLeibniz} we obtain that \( L \to L_4(\alpha) , L_{5} \). Thus, it must be that \(
    \gamma_{2,3}^{2} = 0 \), which means that \( e_2 e _3 = 0 \). Since \( e_3 \) was arbitrary, we may
    assume that \( e_2 e_{i} = 0 \) for \( i \neq 1 ,2 \) and thus we have that \( e_2 \in \text{Ann}(L)
    \).

    Therefore, we have the following multiplcation table
    \begin{align*}
        &e_1e_1 = e_2, \quad e_2 e_{i} = e_{i} e_2 = 0 ,\quad 1 \leq i \leq n,\\
        &e_{1} e_i = - e_i e_1 = \sum_{s = 1}^{n} \gamma_{1,i}^{s} e_{s},\quad  3 \leq i \leq n,\\
        &e_i e_j = \sum_{s = 1}^{n} \gamma_{i,j}^{s} e_{s}, \quad 3 \leq i, j \leq n.
    \end{align*}

     Using the Leibniz identity
     \begin{displaymath}
         0 = (e_1 e_1) e_{i} - (e_{1} e_{i}) e_1 + e_1 (e_{i} e_1) = e_2 e_{i} - \left( \sum_{s = 1}^{n}
         \gamma _{1,i}^{s} e_{s}\right) e_1 - e_1 \left( \sum_{s = 1}^{n} \gamma_{1,i}^{s} e_{s} \right) = -
         2 \gamma_{1,i}^{1} e_2
     \end{displaymath}
     we obtain that \( \gamma_{1,i}^{1} =0\) for \( 3 \leq i \leq n\). Furthermore, if \(\gamma_{1,i}^{2}
     \neq 0\) or \(\gamma_{i,i}^{2} \neq 0\) for some \( 3 \leq i \leq n \), then by Lemma
     \ref{nilpotentLeibniz}, we have that \( L \to L_{4}(\alpha), L_{5} \). Therefore, we may assume
     \(\gamma_{1,i}^{2} = 0\) and \(\gamma_{i,i}^{2} = 0\) for any  \( 3 \leq i \leq n\).

    Moreover, if \(\gamma_{i,j}^{2} \neq 0\) for some \(i, j\) then taking the basis change \( e_{1}' =
    e_{1} + e_j\), we obtain that
    \begin{align*}
        &e_1 e_1 = e_2 + \sum_{i=3}^{n} (*) e_{i}, \quad e_{i} e_1 = \gamma_{i,j}^{2}e_2 + \sum_{i=3}^{n}
        (*) e_{i}, \\
        &e_i e_1 = \gamma_{j,i}^{2} e_2 + \sum_{i=3}^{n} (*) e_{i}, \quad e_{i} e_i =  \sum_{i=3}^{n} (*)
        e_{i}.
    \end{align*}
    If we apply Lemma \ref{nilpotentLeibniz} on the indices \( 1,i,2 \), we see that \( L \to L_4(\alpha), L
    _{5}\). Therefore we can suppose \(\gamma_{i,j}^{2} = 0\) for any \(3 \leq i, j \leq n\).

    By applying the Leibniz identity once again we see that
    \begin{displaymath}
        0 = (e_{i} e_{j}) e_1 - (e_{i} e_1) e_{j} - e_{i} (e_{j} e_1) = \gamma_{i,j}^{1} e_2 +
        \sum_{i = 1, i \neq 2}^{n} (*) e_{i},
    \end{displaymath}
    which means that \( \gamma_{i,j}^1 =0\) for \( 3 \leq i , j \leq n\).

    Therefore, we obtain that any Leibniz algebras $L$ either degenerates to \( L_{4}(\alpha), L_{5} \) or
    \(r_n\) or is a decomposed algebra with ideals \(M_1= \{e_1, e_2\}\) and \(M_2= \{e_3, e_4, \dots,
    e_n\}\). If \( M_{2}\) were trivial, then our only nontrivial multiplication would be \( e_1 e_1 = e_2
    \). This is a contradiction, as we assumed that \( L \not \iso \lambda_{2} \). Thus, \( M_2 \) is not
    trivial, and therefore \( M_2 \) degenerates to an algebra of level one: \( \lambda_{2} , n_{3}^{-} \)
    or \( p_{n}^{-} \). This means that \( L \) degenerates to \( \lambda_{2} \oplus \lambda_{2} \), \(
    \lambda_2 \oplus n_{3}^{-} \), or \( \lambda_{2} \oplus p_{n}^{-} \). By Example \ref{example1}, we
    conclude that \( L \) degenerates to either \( L_{4}(\alpha) \) or \( L_{5} \).
\end{proof}

\begin{theorem}
    Let \(L\) be an \(n\)-dimensional Leibniz algebra of level two. Then \( L \) is isomorphic to one of the
    following pairwise non-isomorphic algebras:
    \begin{displaymath}
        L_4(\alpha)\oplus a_{n-3}, \quad L_{5}\oplus a_{n-3} , \quad r_{n}, \quad \alpha \in \mathbb{C}
    \end{displaymath}
\end{theorem}
\begin{proof}
    Due to Theorem \ref{mainProof}, it is sufficient to prove that these algebras do not degenerate to each
    other. To facilitate this, we compute the dimensions of right annihilator and derivations of these
    algebras and we call upon the following table:
    \begin{alignat*}{2}
        &\dim\rann(L_{5}) = n - 2, \qquad && \dim \text{Der}(L_4(0)) = n^2 - 3n + 4,\\
        & \dim\rann(L_{4}(\alpha)) = n - 2, \ \alpha
        \neq 0, \qquad && \dim \text{Der}(r_{n}) = (n - 1)^2 = n^2 - 2n + 1,\\
        & \dim\rann(r_{n}) = n - 1. &&
    \end{alignat*}
    We first note that \( L_{4}(\alpha) \) and \( L_{5} \) cannot degenerate to \( r_{n} \), as \(
    L_4(\alpha) \) and \( L_{5} \) are nilpotent and \( r_{n} \) is not. We also see that \( r_{n} \) does
    not degenerate to \( L_{4}(\alpha) \) (\( \alpha \neq 0 \)) or \( L_{5} \), as dimension of right
    annihilator of \( r_{n} \) is more than dimensions of right annihilators of \( L_{4}(\alpha) \) (\(
    \alpha \neq 0 \)) and \( L_{5} \).  Additionally, since for \( n \geq 4 \), the dimension of derivations
    of \( r_{n} \) is more than the dimension of derivations of \( L_{4}(0) \), we have that \( r_{n}
    \not\to L_{4}(0) \).  Lastly, we see that \( L_{4}(\alpha) \not\to L_{5} \) and that \( L_{5} \not\to
    L_{4}(\alpha) \) by the following paper \cite{nilpotentLeibnizFive}.
\end{proof}

\begin{remark}
    We note that in the context of left Leibniz algebras, the following algebra replaces the algebra \(
    r_{n} \) as a Leibniz algebra of level two:
    \begin{displaymath}
        \ell_{n} : \qquad e_1 e_{i} = e_{i}, \qquad 2 \leq i \leq n.
    \end{displaymath}
\end{remark}

\subsection{Nilpotent algebras of level two}

Working in the variety of $n$-dimensional nilpotent algebras $\text{Nil}_n ( \mathbb{C} )$ will allow us to exclude certain products from our multiplication tables, in particular all products of the form $xy = x$.

 \begin{theorem}\label{nilpotentclass}
Any $n$-dimensional $(n \geq 5)$ nilpotent algebra of level two is isomorphic to one of the following algebras:
$$\begin{array}{rlllll}
n_{5, 1}   : & e_1 e_2 = e_5, &  e_3e_4 = e_5, & e_2 e_1 = -e_5, &  e_4e_3 = - e_5;\\
n_{5, 2}  : & e_1 e_2 = e_4 , & e_1 e_3 = e_5,& e_2 e_1 = - e_4 , & e_3 e_1 = -e_5;\\
L_4 (\alpha)   :& e_1 e_1 = e_3, & e_2 e_2 = \alpha e_3, & e_1 e_2 = e_3;\\
L_5  : & e_1 e_1 = e_3 , & e_1 e_2 = e_3, & e_2 e_1 = e_3.
\end{array}
$$
\end{theorem}

\begin{proof}  Our overall strategy is to look separately at antisymmetric and non-antisymmetric cases, and then at the way products fall into the square of the algebra $A^2$.

\textbf{Case {1.}} First we assume that $A \in \text{Nil}_n
(\mathbb{C} )$ is non-antisymmetric.

\textbf{Case {1.1.}} Assume that $\text{dim} A^2 = 1$, then we
assume that $A^2 = \{e_n\}$ and have the following multiplication
\[A :\left\{
\begin{array}{lll}
e_1e_1 = e_n, & e_i e_j = \alpha_{i,j} e_n , & 2 \leq i,j
\leq n-1, \\
e_1 e_j = \alpha_{1,j} e_n , & e_j e_1 = \alpha_{j,1} e_n ,& 2 \leq
j \leq n-1.
\end{array} \right.
\]

If there exist $i$ such that $| \alpha_{1i} - \alpha_{i1} | + |
\alpha_{ii} - \alpha_{1i} \alpha_{i1} | \neq 0$, then by Lemma
\ref{nilpotentLeibniz} we obtain that $L$ degenerates to $L_4
(\alpha)$ or $L_5$.

Now let $|\alpha_{1i} - \alpha_{i1} | + | \alpha _{ii} - \alpha_{1i} \alpha_{i1} | = 0$ for $1 \leq i \leq n$.
 Making the change of basis $$e_1 ' = e_1, \hspace{0.2cm} e_i ' = e_i + \alpha_{1,i} e_1, \hspace{0.4cm} 2 \leq i \leq n-1$$
the multiplication of $A$ simply becomes $$(*) \hspace{0.5cm} e_1
e_1 = e_n , \hspace{0.2cm} e_1 e_i=e_i e_1 =e_i e_i =0,
\hspace{0.2cm} e_i e_j = \alpha_{ij} e_n, \hspace{0.4cm} 2 \leq i, j
\leq n -1.$$ We consider now the subalgebra $M : \{ e_2 , ..., e_n
\}$. Note that $M$ cannot be abelian, since otherwise $(*)$ becomes
an algebra $\lambda _2$, which is level one. Thus we have
the following two subcases, which will complete the case $\text{dim} A^2
= 1$.

\textbf{Case {1.1.1.}} Assume that $M$ is Lie. Since $M$ is also not
abelian, we are free to choose $\alpha_{2,3} = 1,$ $\alpha_{3,2} = -1$. Taking the
degeneration
\[
\begin{cases}
g_t(e_1) = t^{-2} e_n, &
g_t(e_2) =2t^{-1}e_2 - t^{-2} e_n,\\
g_t(e_3) =t^{-1} e_1 - t^{-2} e_n,&
g_t(e_n) =t^{-2} e_3,\\
g_t(e_i) = t^{-2} e_i, & 4 \leq i \leq n-1,
\end{cases}
\]
we obtain that $A$ degenerates to $L_4(\frac 1 4)$.

\textbf{Case {1.1.2.}} Assume that $M$ is non-Lie. Then  we may
assume $e_2 e_2 \neq 0,$ moreover $e_2 e_2 = e_n$. Taking the
degeneration
\[
g_t(e_1) = e_1, \quad
g_t(e_2) =e_2,\quad
g_t(e_n) = e_n,\quad
g_t(e_i) = t^{-1} e_i, \quad  3 \leq i \leq n-1,
\]
we obtain that $A$ degenerates to $L_5$.

\textbf{Case {1.2.}} We now assume that $\text{dim} A^2 \geq 2$. Let
$A= \{e_1, \ldots , e_n \}$, and $A^2 = \{e_{k+1} , \ldots , e_n
\}$. We consider five logically exhaustive cases in which the
products $e_ie_1$, $e_1 e_i$ fall in the square $A^2$ in different
ways. It is obvious that we may always assume $e_1 e_1 = e_{k+1}$.

\textbf{Case {1.2.1.}} Assume $e_i e_1 \in span\{e_{k+1}\} $
for $1 \leq i \leq k$, and let $e_1 e_2 \notin span\{ e_{k+1}
\}$, so that we have the multiplication:
$$e_1 e_1 = e_{k+1}, \hspace{0.3cm} e_1 e_2 = e_{k+2}, \hspace{0.3cm} e_i e_1 = \gamma_{i,1} ^{k+1} e_{k+1}, \hspace{0.2cm} 2 \leq i \leq k, $$
$$ e_i e_j = \sum_{\ell = k+1} ^n \gamma_{i,j} ^{\ell} e_{\ell}, \hspace{0.6cm} 2 \leq i , \hspace{0.1cm} j \leq k.$$

Applying the Lemma \ref{nilpotentLeibniz} for the elements
$(\gamma_{1,1} ^{k+2}, \gamma_{1,2} ^{k+2}, \gamma_{2,1} ^{k+2},
\gamma_{2,2} ^{k+2}) = (0,0,1, \gamma_{2,2} ^{k+2})$ we obtain that
$A$ degenerates to $L_4 (\alpha)$ or $L_5$.

\textbf{Case {1.2.2.}} Let there exist $i$ such that $e_ie_1 \notin
span \{ e_{k+1} \}$. Without loss of generality, we may put
$e_2e_1 = e_{k+2}$, so that our products are
$$ e_1 e_1 = e_{k+1}, \hspace{0.6cm} e_2 e_1 = e_{k+2}, \hspace{0.6cm} e_i e_1 = \sum_{\ell = k+1} ^n \gamma_{i,1} ^{\ell} e_{\ell}, \hspace{0.6cm} 3 \leq i  \leq k,$$
$$ e_1 e_i = \sum_{\ell = k+1} ^n \gamma_{1,i} ^{\ell} e_{\ell}, \hspace{0.3cm} 2 \leq i  \leq k, \hspace{0.8cm}
e_i e_j = \sum_{\ell = k+1} ^n \gamma_{i,j} ^{\ell} e_{\ell}, \hspace{0.3cm} 2 \leq i , \hspace{0.1cm} j \leq k.$$

If $\gamma_{1,2} ^{k+2} \neq -1,$ then  applying
the Lemma \ref{nilpotentLeibniz} for the $(\gamma_{1,1}
^{k+2}, \gamma_{1,2} ^{k+2}, \gamma_{2,1} ^{k+2}, \gamma_{2,2}
^{k+2}) = (0, \gamma_{1,2} ^{k+2}, 1, \gamma_{2,2} ^{k+2})$ we obtain that $L$ degenerates to
$L_4 (\alpha)$ or $L_5$.

If $\gamma_{1,2} ^{k+2} = -1,$ then taking the change of basis
$$e_2 ' =  e_2 + \eta e_1, \ e_{k+2} ' =  e_{k+2} + \eta e_{k+1}, \  e_i' = e_i , \ 1 \leq i (i \neq 2, k+2) \leq n,$$
we obtain that
$$ e_1' e_1' = e_{k+1}', \hspace{0.6cm} e_2' e_1' = e_{k+2}',
\hspace{0.6cm} e_1' e_2' =  (2\eta + \gamma_{1,2} ^{k+1})e_{k+1}' -
e_{k+2}' + \sum_{\ell = k+3} ^n \gamma_{1,i} ^{\ell} e_{\ell},
$$
$$e_i' e_1' = \sum_{\ell = k+1} ^n \gamma_{i,1}
^{\ell} e_{\ell}', \hspace{0.6cm} e_1' e_i' = \sum_{\ell = k+1}^n
\gamma_{1,i} ^{\ell} e_{\ell}', \hspace{0.6cm} 3 \leq i  \leq k,$$
$$e_i' e_j' = \sum_{\ell = k+1} ^n  \gamma_{i,j} ^{\ell} e_{\ell}' \hspace{0.6cm} 2 \leq i , \hspace{0.1cm} j \leq k.$$

Taking the value of $\eta$ such that ${\gamma_{1,2} ^{k+1}}' =  2\eta +
\gamma_{1,2} ^{k+1} \neq 0$ we apply the Lemma
\ref{nilpotentLeibniz} for the elements $({\gamma_{1,1} ^{k+1}}',
{\gamma_{1,2} ^{k+1}}, {\gamma_{2,1} ^{k+1}}', {\gamma_{2,2}
^{k+2}}') = (1, {\gamma_{1,2} ^{k+1}}', 0, {\gamma_{2,2} ^{k+2}}')$
and obtain that $A$ degenerates to $L_4 (\alpha)$ or $L_5$.

\textbf{Case {1.2.3.}} Now suppose that $e_i e_1 , e_1 e_i \in
span \{ e_{k+1} \}$ for $1 \leq i \leq k$ and there exist $j$
such that $e_j e_j \notin span \{ e_{k+1} \}$. In this case
without loss of generality, we may suppose $e_2 e_2 = e_{k+2}$, so
that our multiplication becomes
$$ e_1 e_1 = e_{k+1}, \ e_2 e_2 = e_{k+2}, \ e_i e_1 = \gamma _{i,1} ^{k+1} e_{k+1}, \ e_1 e_i = \gamma _{1,i} ^{k+1} e_{k+1}, \hspace{0.3cm} 2 \leq i \leq k,$$
$$e_i e_j = \sum_{\ell = k+1} ^n \gamma _{i,j} ^{\ell} e_{\ell}, \hspace{0.3cm} 3 \leq i, \hspace{0.2cm} j \leq k.$$

If $(\gamma _{1,2} ^{k+1}, \gamma _{2,1} ^{k+1}) \neq (0,0)$ then
applying the Lemma \ref{nilpotentLeibniz} for the $(\gamma_{1,1}
^{k+1}, \gamma_{1,2} ^{k+1}, \gamma_{2,1} ^{k+1}, \gamma_{2,2}
^{k+1}) = (1, \gamma_{1,2} ^{k+1}, \gamma_{2,1} ^{k+1}, 0)$ we
obtain that $A$ degenerates to $L_4 (\alpha)$ or $L_5$.

If $\gamma _{1,2} ^{k+1} = \gamma _{2,1} ^{k+1} = 0,$ then taking
the degeneration
\[g_t(e_1) = t^{-1}e_1, \quad g_t(e_2) = t^{-1}e_2, \qquad g_t(e_i) = t^{-2}e_i, \quad 3 \leq i \leq n,\]
we degenerate to the algebra
$$\lambda_2 \oplus \lambda_2 : e_1 e_1 = e_{k+1}, \hspace{0.2cm} e_2 e_2 = e_{k+2}.$$

By Example \ref{example1}, we obtain that algebra $\lambda_2 \oplus \lambda_2$
degenerates to the algebra $L_5.$

\textbf{Case {1.2.4.}} Now we suppose that $e_i e_i , e_1 e_i, e_i
e_1 \in span \{ e_{k+1} \} $ for $1 \leq i \leq k$. Let there
exist $i, j \ (2 \leq i,j \leq k)$ such that $e_{i} e_{j} \notin
span \{ e_{k+1} \}$. Without loss of generality we can suppose
$i=2, j=3$ moreover, if  $e_{2} e_{3} + e_{3} e_{2} \neq 0$, then
taking $e_2' = e_{2} + e_{3}$, we get that $e_2'e_2' = e_2e_2 +
e_2e_3 + e_3e_2 + e_3e_3 \notin \{ e_{k+1} \}$ which have
the situation of Case 1.2.3. Therefore, we may suppose
$$e_1e_1 = e_{k+1}, \ e_1e_2 = \alpha_{1,2}e_{k+1}, \ e_2e_1 =
\alpha_{2,1}e_{k+1}, \  e_2e_2 = \alpha_{2,2}e_{k+1}, $$
$$e_1e_3 = \alpha_{1,3}e_{k+1}, \ e_3e_1
= \alpha_{3,1}e_{k+1}, \  e_3e_3 = \alpha_{3,3}e_{k+1}, $$
$$e_2e_3 = e_{k+2}, \ e_3e_2 = - e_{k+2}.$$

Now applying the degeneration
\begin{displaymath}
\begin{cases}
g_t(e_1) = t^{-2} e_1, & g_t(e_{2}) = t^{-3} e_2, \\  g_t(e_{3}) = t^{-3} e_{3}, & g_t(e_{k+1}) = t^{-4} e_{k+1},
\\ g_t(e_{i}) = t^{-5} e_i, & 4 \leq i \leq n.
\end{cases}
\end{displaymath}

we obtain the products
$$e_1 e_1 = e_{k+1}, \hspace{0.2cm} e_2 e_3 = e_{k+2}, \hspace{0.2cm} e_3 e_2 = - e_{k+2}.$$
which by Example \ref{example1} degenerates to the algebra $L_5$.

\textbf{Case {1.2.5.}} Now we suppose that $e_i e_j \in span \{
e_{k+1} \} $ for all $i,j \ (1 \leq i, j \leq k)$. Thus we have
$$e_1e_1 = e_{k+1}, \quad e_ie_j = \alpha_{i,j}e_{k+1}.$$
Since $dim A^2 \geq 2,$ then we have that there exist $i \ (2 \leq i
\leq k),$ such that $e_ie_{k+1}$ or $e_{k+1}e_i$ is non zero, which
we can suppose as $e_{k+2}.$

Let $i=1,$ then we have $e_1e_1 = e_{k+1}$ , $e_1e_{k+1} = e_{k+2}$.

If  $e_{k+1}e_1 + e_1e_{k+1} \neq 0,$ then using the Lemma
\ref{nilpotentLeibniz} for the basis elements $\{e_1, e_{k+1},
e_{k+2}\}$ we have that $A$ degenerates to the algebra $L_4(\alpha)$
or $L_5.$

If  $e_{k+1}e_1 = - e_1e_{k+1} = - e_{k+2}$ then making the change
$e_{k+1}' = e_{k+1}-e_{k+2},$ we have that
$$e_1 e_1 = e_{k+1} + e_{k+2}, \quad e_1e_{k+1} = e_{k+2}, \quad e_{k+1}e_1 = -e_{k+2}.$$
Again applying the Lemma \ref{nilpotentLeibniz} for the basis
elements $\{e_1, e_{k+1}, e_{k+2}\},$ i.e. $(\gamma_{1,1}^{k+2},
\gamma_{1,k+1}^{k+2}, \gamma_{k+1, 1}^{k+2},\gamma_{k+1,k+1}^{k+2})
= (1, 1, -1, \gamma_{k+1,k+1}^{k+2} )$ we have that $A$ degenerates
to the algebra $L_4(\alpha)$ or $L_5.$

Let $i\neq 1,$ then we can suppose $i=2$ and we have $e_1e_1 =
e_{k+1}$ , $e_2e_{k+1} = e_{k+2}$.

Similarly to the case $i=1$ if  $e_{k+1}e_2 + e_2e_{k+1} \neq 0,$
then using the Lemma \ref{nilpotentLeibniz} for the basis elements
$\{e_2, e_{k+1}, e_{k+2}\}$ we have that $A$ degenerates to the
algebra $L_4(\alpha)$ or $L_5.$

If  $e_{k+1}e_2 = - e_2e_{k+1} = - e_{k+2}$ then making the change
$e_{k+1}' = e_{k+1}-e_{k+2},$ we have that
$$e_1 e_1 = e_{k+1} + e_{k+2}, \quad e_1e_{k+1} = e_{k+2}, \quad e_{k+1}e_1 = -e_{k+2}.$$
Again applying the Lemma \ref{nilpotentLeibniz} for the basis
elements $\{e_1, e_{k+1}, e_{k+2}\},$ i.e. $(\gamma_{1,1}^{k+2},
\gamma_{1,k+1}^{k+2}, \gamma_{k+1, 1}^{k+2},\gamma_{k+1,k+1}^{k+2})
= (1, 1, -1, \gamma_{k+1,k+1}^{k+2} )$ we have that $A$ degenerates
to the algebra $L_4(\alpha)$ or $L_5.$

\textbf{Case {2.}} Let $A \in \text{Nil}_n ( \mathbb{C} )$ be antisymmetric.
It should be noted that if $dimA^2 =1,$ then we have that $A^3=0.$ Thus, $A$ is a Lie algebra. In \cite{leveltwo} it is shown that any nilpotent Lie algebra with condition $dimA^2 =1,$  $A^3=0$ degenerate to algebra $n_{5,1}.$

 Therefore we consider case $\dim A^2 \geq 2$. Assume that
 $\{ e_1, e_2, \dots , e_k, e_{k+1},\dots, e_n\}$ be a basis of $A$,
 and $\{e_{k+1}, e_{k+2}, \dots, e_n\}$ be a basis of $A^2$.

 Then, without loss of generality, we can assume $e_1 e_2 = e_{k+1},$ $e_2 e_1 = - e_{k+1}.$

Below, we show that it may always be assumed $$e_1 e_2 = e_4, \
e_1 e_3 = e_5.$$

\begin{itemize}
\item Let there exists $i_0$ such that $e_1 e_{i_0} \notin
span\langle x_{k+1}\rangle.$ Then taking
$$e'_1= e_1, \  e'_2 = e_2, \ e'_3 = e_{i_0}, \ e'_4 = e_{k+1}, \ e'_5 = e_1 e_{i_0} $$
we obtain $e'_1 e'_2 = e'_4, \ e'_1 e'_3 = e'_5.$

\item Let $e_1 e_{i} \in span \{ e_{k+1}\}$ for all $3 \leq i \leq k$ and
there exists some $i_0$ such that $e_2 e_{i_0} \notin
span\langle x_{k+1}\rangle.$ Due to symmetrically of $e_1$ and
$e_2,$ similarly to the previous case we can choose a basis
$\{e'_1, e'_2, \dots, e'_n\}$ with condition $e'_1 e'_2 = e'_4,
\ e'_1 e'_3 = e'_5.$

\item Let $x_1 x_{i},\  x_2 x_{i} \in span \{ e_{k+1}\}$ for all $3 \leq i \leq
k.$ We set $e_1 e_{i} = \alpha_{i}e_{k+1}$ and $e_2 e_{i} =
\beta_{i}x_{k+1}.$ Let $e_{i_0}$ and $e_{j_0}$ be generators of
$A$ such that $e_{i_0} e_{j_0} \notin span \{
e_{k+1}\}.$ Since $dim A^2 \geq 2$ one can assume $e_{i_0}
e_{j_0} = e_{k+2}.$

Putting $$e'_1 = e_1 + Ae_{i_0},  \ e'_2= e_2,  \ e'_3 = e_{j_0},
\ e'_{4} = (1-A\beta_{i_0})e_{k+1}, \ e'_5 = Ae_{k+2} +
\alpha_{i_0} e_{k+1}$$ with $A(1-A\beta_{i_0}) \neq 0,$ we deduce
$e'_1 e'_2 = e'_4, \ e'_1 e'_3 = e'_5.$

\item Let $e_i e_{j} \in span \{ e_{k+1}\}$ for all $1 \leq i,j \leq
k.$ Then for some $i_0$ we have $e_{i_0} e_{k+1} \neq 0.$
Without loss of generality, one can assume $e_1 e_{k+1} =
e_{k+2}.$
    \begin{itemize} \item If $k\geq 3,$ then setting
    $$e'_1 = e_1,  \ e'_2= e_2,  \ e'_3 =
    e_{3}+e_{k+1}, \ e'_{4} = e_{k+1}, \ e'_5 = e_{k+2} + \alpha_{1,3} e_{k+1},$$ we obtain $e'_1 e'_2 = e'_4, \ e'_1 e'_3 =
e'_5.$

    \item If $k=2,$ then we have $e_1 e_2=e_3, \ e_1e_3=e_4.$
It is not difficult to obtain that
    $e_1e_4=e_5$ or $e_2e_3=e_5$ (because of $n \geq 5$). Indeed, taking $$e'_1 = e_1,  \ e'_2= e_2,  \ e'_3 =  e_{4}, \ e'_{4} = e_{3}, \ e'_5 = e_{5}$$ in the case of $e_1e_4=e_5$ and
    $$e'_1 = -e_3,  \ e'_2= e_1,  \ e'_3 =e_{4}, \ e'_{4} = e_{2}, \ e'_5 = e_{5}$$ in the case of $e_2e_3=e_5,$ we derive the products $e_1 e_2 = e_4, \ e_1 e_3 = e_5.$

    \end{itemize}
\end{itemize}

Thus, there exists a basis $\{e_1, e_2, e_3, \dots, e_n\}$ of $A$
with the products $$e_1 e_2 = e_4, \ e_1 e_3 = e_5.$$

Note that $A$ degenerates to the algebra with multiplication:
$$e_1 e_2 = e_4, \ e_1 e_3 = e_5, \ e_2 e_3 = \gamma_4 x_4 + \gamma_4 x_5$$
via the following degeneration:
\[g_t: \left\{\begin{array}{llll} g_t(e_1)= t^{-2} e_1,& g_t(e_2)= t^{-2} e_2,& g_t(e_3)= t^{-2}
e_3,&\\
g_t(e_4)= t^{-4} e_4,& g_t(e_5)= t^{-4} e_5,&
g_t(e_{i})=t^{-3}e_{i},  & 6 \leq i \leq n. \end{array}\right. \]

From the change of basis $e_2'=e_2 - \gamma_5e_1,$ $e_3'=e_3 +
\gamma_4e_1,$ we obtain that this algebra is isomorphic to
$n_{5,2} \oplus \mathfrak{a}_{n-5}.$
\end{proof}

\end{document}